\documentclass[12pt,reqno]{amsart}
\usepackage{amsfonts}
\usepackage{amsthm}
\usepackage{amsmath}
\usepackage{amscd}
\numberwithin{equation}{section}
\usepackage{hyperref}
\hypersetup{nesting=true,debug=true,naturalnames=true}
\usepackage{graphicx,amssymb,upref}
\usepackage{enumerate}
\usepackage[all]{xy}
\usepackage{color}
\usepackage{mathrsfs}
\usepackage{caption}
\usepackage{paralist}
\usepackage{color}
\usepackage{float}
\usepackage{tikz,pgf}
\usetikzlibrary{calc}
\usetikzlibrary{intersections,decorations.markings}
\usepackage{tikz-cd}
\usepackage{lineno}

\theoremstyle{definition}
\newtheorem{theorem}{\bf Theorem}[section]

\newtheorem{lemma}[theorem]{\bf Lemma}

\theoremstyle{definition}

{
\newtheorem*{Ack}{\bf Acknowledgements}
}
\newcommand{\mm}[1]{\mathrm{#1}}
\newcommand{\mb}[1]{\mathbb{#1}}

\makeatletter
\@namedef{subjclassname@2020}{\textup{2020} Mathematics Subject Classification}
\makeatother

\begin{document}

\title[{\fontsize{7}{7}\selectfont  Loop Space Splittings of Sphere Bundles}]{Loop Space Splittings of Sphere Bundles over Highly Connected Poincar\'e Complexes}

\author[{\fontsize{7}{7}\selectfont Wen Shen}]{Wen Shen}

\address{College of Mathematics and Physics, Wenzhou University, Wenzhou, P.R.China}
\email{shenwen121212@163.com}

	\begin{abstract}
Let $m > n \ge 2$, and let $N$ be an $(n-1)$-connected $2n$-Poincar\'e complex. In this paper, we establish sufficient conditions under which the loop space of the total space $M$ of the sphere bundle $S^{m-1} \to M \to N$ (associated to a rank-$m$ real vector bundle over $N$) splits as a product of the loop spaces of  $N$ and $S^{m-1}$.

\end{abstract}

\maketitle

\section{Introduction}

In the realm of algebraic and geometric topology, manifold classification stands as one of the most fundamental and far-reaching central problems. 
To classify manifolds, we first pick a suitable equivalence relation which describes a unique topological or smooth structure. The most widely studied among these relations are diffeomorphism, homeomorphism, and homotopy equivalence.
 Beyond these standard frameworks, there exists a more specialized and increasingly prominent classification paradigm: loop decomposition equivalence. Formally, two manifolds \(X\) and \(Y\) are deemed equivalent under this relation if their loop spaces satisfy \(\Omega X\simeq \Omega Y\), where \(\Omega\) denotes the loop functor that maps a topological space \(X\) to the space of all based loops in \(X\), and a continuous map \(f:X\to Y\) to the induced map on loops. The study of loop space equivalences not only reveals hidden homological and homotopical connections between manifolds but also provides a powerful lens for computing homotopy invariants that are otherwise intractable.

A cornerstone of manifold classification lies in the study of highly connected manifolds, a class of spaces with vanishing homotopy groups in low dimensions. This line of inquiry has a long and illustrious history, with foundational contributions from C.T.C. Wall \cite{Wall1962,Wall1967}.
Up to loop decomposition, Beben and Theriault \cite{BeTh} proved that \((n-1)\)-connected \(2n\)-manifolds are equivalent if and only if their $n$-th Betti numbers are equal. Parallel to this work, the loop decomposition problem for the odd-dimensional counterpart, namely \((n-1)\)-connected \((2n+1)\)-manifolds, has been tackled by two independent teams: Beben and Wu \cite{BW}, and Huang and Theriault \cite{HT2022}. A key virtue of these loop decomposition theorems is their computational utility: they reduce the problem of computing the homotopy groups of highly connected manifolds to the computation of homotopy groups of spheres and Moore spaces. From a different perspective, Samik Basu and Somnath Basu \cite{Basu,BasuBasu} have presented related computations.

 Sphere bundles occupy a central role in topology, serving as a bridge between manifold theory, vector bundle theory, and homotopy theory. The homotopy classification of sphere bundles has a rich heritage dating back to the work of James and Whitehead \cite{JW1,JW2}. 
In recent years, the study of loop decompositions has been extended to sphere bundles over \(4\)-manifolds \cite{Huang23,Huang25}.
A striking result from \cite{Huang25} asserts that the loop spaces of the sphere bundles of most real vector bundles over simply connected closed \(4\)-manifolds split into a product of simpler homotopy types. This breakthrough not only generalizes earlier loop decomposition results for manifolds but also raises compelling questions about the extent to which such splittings hold for sphere bundles over more general spaces.
Inspired by this work on sphere bundles over \(4\)-manifolds, we turn attention to a natural and far-reaching generalization: sphere bundles associated to real vector bundles over highly connected manifolds. 

Now we present the main result:
\begin{theorem}\label{mainth}
Let $m> n\ge 2$.
	Let $N$ be an $(n-1)$-connected $2n$-Poincar\'e complex.  
	Let 
	$$S^{m-1}\to M\to N$$
	be the sphere bundle of a rank-$m$ real vector bundle over $N$. Then 	$$\Omega M\simeq \Omega N\times \Omega S^{m-1}$$
	 if $n,m$, and $d=\mm{rank}(H^n(N;\mb{Z}))$ satisfy one of the following conditions 	
	\begin{enumerate}
		\item $n\ge 3$, $d\ge 2$.
		\item $n=2$, $(m,d)\ne (3,0)$ and $(4,0)$.
		\item $m>2n$, $d=0$.
	\end{enumerate}
\end{theorem}

For $d=0$, the complex $N$ as in Theorem \ref{mainth} is homotopy equivalent to $S^{2n}$. For $d=1$, Adams'  result \cite{Adams1960} implies that $n$ must be $2,4$ or $8$. Specifically, for $n=2$, $N\simeq \mb{CP}^2$; however, for $n=4$ or $8$, $N$ is not necessarily homotopy equivalent to  $\mb{HP}^2$ or $\mb{OP}^2$ respectively. This constitutes a key obstruction to the study of the cases $(n,d)=(4,1)$ and $(n,d)=(8,1)$. 
For $d\ge 2$, $\Omega N$ has already been established in \cite[Theorem 1.4]{BeTh}. Notably, for $n\ne 2,4,8$, Adams'  result \cite{Adams1960} and Poincar\'e duality imply $d\ne 1$.

As an immediate consequence of Theorem \ref{mainth}, we obtain an isomorphism of homotopy groups:
$$\pi_\ast(M)\cong \pi_\ast(S^{m-1})\oplus \pi_\ast(N)$$

Among the results in Theorem \ref{mainth}, Item (2) coincides with the main theorem of \cite{Huang25}. For Item (3), we follow the idea of \cite{Huang25} to complete the proof. Notably, Item (1) is proven from a novel and unified perspective, which distinguishes it from the existing arguments in the literature. Moreover, the argument for Item (1) also recovers the cases for $n=2$ and $d \ge 2$ as in \cite{Huang25}.



The structure of this paper is organized as follows: In Section \ref{Sevec}, we first investigate vector bundles over highly connected complexes and their classifying maps, presenting key observations that lay the foundation for proving Theorem \ref{mainth}. Subsequently, in Section \ref{Sepro}, we provide the detailed proofs of Items (1) and (3)  of Theorem \ref{mainth}.

\section{Vector bundles over highly connected complexes}\label{Sevec}

First, we recall a classic result in homotopy theory that is fundamental to our subsequent arguments:

\begin{lemma}
	For any connected space $X$ and any integer $k\ge 2$, there exists a $(k-1)$-connected space $X\langle k\rangle$ with a fibration $p:X\langle k\rangle\to X$ that induces isomorphisms on $\pi_i$ for $i\ge k$.
\end{lemma}
\begin{proof}
	This lemma is a direct consequence of the Whitehead tower construction for the space $X$.
\end{proof}

Let $N$ be an $(n-1)$-connected $2n$-Ponincar\'e complex satisfying $H^n(N;\mb{Z})=\mb{Z}^d$. By the minimal cell structure \cite[Proposition 4.1]{Wall1965}, $N$ has the following homotopy type
$$N=(\vee_{i=1}^d S^n )\cup_\alpha e^{2n}$$
Let $\bar N=\vee_{i=1}^d S^n$ denote the $2n-1$-skeleton of $N$. 

Recall that a rank-$m$ real vector bundle over $N$ admits a classifying map $f:N\to \mm{BO}_m$ where $\mm{BO}_m$ is the classifying space of the orthogonal group $\mm{O}_m$.

\begin{lemma}\label{liftBOn}
	 Any map $f:N\to \mm{BO}_m$ admits a lift $$g:N\to \mm{BO}_m\langle n\rangle$$
	along the fibration $p:\mm{BO}_m\langle n\rangle\to \mm{BO}_m$.
\end{lemma}
\begin{proof} 
	Let $F$ be the homotopy fiber of $p:\mm{BO}_m\langle n\rangle\to \mm{BO}_m$. Since $p$ induces isomorphisms on $\pi_i$ for $i\ge n$, we have $\pi_i(F)=0$ for $i\ge n$. Note the cell structure of $N$. We can use the obstruction theory \cite[p.418]{Hatcher} to prove the lemma.	   
\end{proof}

Note that the space $\mm{BO}_m\langle n\rangle$ is $(n-1)$-connected. By the Hurewicz theorem, we have  $\pi_n(\mm{BO}_m\langle n\rangle)\cong H_n(\mm{BO}_m\langle n\rangle)$. Hence, the homotopy class of any map $\iota:S^n\to \mm{BO}_m\langle n\rangle$ is uniquely determined by its induced homomorphism on integral homology. We now generalize this observation to maps defined on wedge sums of 
$n$-spheres, which is critical for simplifying the structure of our classifying maps:
\begin{lemma}\label{transform}
	Let $m>n\ge 2,$ $s>0$. For any map $r:\vee_{i=1}^s S^n\to \mm{BO}_m\langle n\rangle$, there exists a homotopy equivalence $h:\vee_{i=1}^s S_i^n\to \vee_{i=1}^s S_i^n$ such that 
	the induced homomorphism 
	$$(r\circ h)_\ast:H_n(\vee_{i=1}^s S_i^n;\mb{Z})\to H_n(\mm{BO}_m\langle n\rangle;\mb{Z})$$ 
	sends the generator $a_i$ of $H_n(S^n_i;\mb{Z})$ to zero for all $i\ge 2$.
		 Moreover, the homotopy class of $r\circ h$ is determined by the homology homomorphism.
\end{lemma} 
\begin{proof}
	First, $\pi_n(\mm{BO}_m\langle n\rangle)\cong\pi_n(\mm{BO}_m)$.
	By $m>n\ge 2,$ $\pi_n(\mm{BO}_m)\cong \pi_n(\mm{BO})$. By Bott periodicity, $\pi_n(\mm{BO})\cong\mb{Z}$, $\mb{Z}_2$ or $0$,
	so the same holds for $\pi_n(\mm{BO}_m\langle n\rangle)$. We proceed by case analysis:
	
Case 1: $\pi_n(\mm{BO}_m\langle n\rangle)\cong 0$. Take $h$ to be the identity map and the lemma is obvious. 
	
 Case 2: $\pi_n(\mm{BO}_m\langle n\rangle)\cong \mb{Z}$. By the Hurewicz isomorphism, we have $H_n(\mm{BO}_m\langle n\rangle;\mb{Z})\cong \mb{Z}$. On the other hand, $ H_n(\vee_{i=1}^s S_i^n;\mb{Z})\cong \mb{Z}^s$ with generators $e_i$ corresponding to the $i$-th sphere $S^n_i$. The induced homomorphism $r_\ast:H_n(\vee_{i=1}^s S_i^n;\mb{Z})\to H_n(\mm{BO}_m\langle n\rangle;\mb{Z})$ is thus a group homomorphism from a free abelian group of rank $s$ to $\mb{Z}$.

By the Hurewicz isomorphism, constructing the homotopy equivalence $h:\vee_{i=1}^s S_i^n\to \vee_{i=1}^s S_i^n$ as in lemma is equivalent to finding a basis of $\mb{Z}^s$ such that $r_\ast$ vanishes on all basis elements except the first one.   

If $r_\ast$ is trivial, we also take $h$ to be the identity map. Thus the lemma is obvious. 

Now we assume that the homomorphism $r_\ast$ is non-trivial. Let $$r_\ast(e_i)=d_i\in \mb{Z}, \quad
D=\mm{gcd}(d_1,d_2,\cdots,d_s).$$  
 By the Smith normal form, there exists a  matrix $P=(p_{i,j})\in \mm{GL}_{s}(\mb{Z})$ such that $P\cdot (d_1,d_2,\cdots,d_s)^{\mm{T}}=(D,0,\cdots,0)^{\mm{T}}$.
 
Let $a_j=\Sigma^s_{i=1}p_{i,j}e_i$, $1\le j\le s$. It is easy to check that 
\begin{itemize}
	\item  $\{a_1,a_2,\cdots,a_s\}$ is a basis of $\mb{Z}^s$;
	\item $r_\ast(a_i)=0$ for $i\ge 2$.
\end{itemize}
Take $h$ to be the self-equivalence of $\vee_{i=1}^s S^n_i$ determined by $P$.
This completes the proof for this case.
	
Case 3: $\pi_n(\mm{BO}_m\langle n\rangle)\cong \mb{Z}_2$. The argument is analogous to Case 2 but simpler. We leave the details to the reader.  
\end{proof}

We now state our main structural theorem for maps from the Poincar\'e complex $N$ to $\mm{BO}_m\langle n\rangle$, which simplifies the homotopy class of such maps by collapsing redundant spheres in the wedge sum:

\begin{theorem}\label{basicob}
Let $m>n\ge 2$.	Let $N\simeq (\vee_{i=1}^d S_i^n )\cup_\alpha e^{2n}$ be a Ponincar\'e complex. 
The homotopy class of any map
$g:N\to \mm{BO}_m\langle n\rangle$ lies in the image of 
the following morphism:
$$[\widetilde N, \mm{BO}_m\langle n\rangle]\stackrel{\mm{Col}^\ast}{\to }[N,\mm{BO}_m\langle n\rangle]$$
induced by the map $\mm{Col}:N\to \widetilde N$,
where $\widetilde N$ is a complex defined via the cofibration sequence
$$\vee_{i=2}^d S^n_i\to N\stackrel{\mm{Col}}{\to} \widetilde N=S^n_1\cup e^{2n}$$
\end{theorem}
\begin{proof}
	By Lemma \ref{transform}, there exists a homotopy equivalence $$h:\vee_{i=1}^d S_i^n \to \vee_{i=1}^d S_i^n $$
	such that $g|_{\vee_{i=1}^d S_i^n }\circ h$ induces a trivial homology homomorphism on $\vee_{i=2}^d S_i^n$. By the Hurewicz theorem, $(g|_{\vee_{i=1}^d S_i^n }\circ h)|_{\vee_{i=2}^d S_i^n}$	is null-homotopic. Furthermore, $h$ can be extended to the homotopy equivalence
	$$\mm{H}:(\vee_{i=1}^d S_i^n )\cup_{h^{-1}\circ\alpha} e^{2n}\to (\vee_{i=1}^d S_i^n )\cup_{\alpha} e^{2n}$$
	satisfying $\mm{H}|_{\vee_{i=1}^d S_i^n}=h$. 
	
	We now consider the complex $(\vee_{i=1}^d S_i^n )\cup_{h^{-1}\circ\alpha} e^{2n}$ and the map $$g\circ \mm{H}: (\vee_{i=1}^d S_i^n )\cup_{h^{-1}\circ\alpha} e^{2n}\to \mm{BO}_m\langle n\rangle$$
	For this new model of $N$, the restriction of $g\circ \mm{H}$ to $\vee_{i=2}^d S^n_i$ is  null-homotopic. By applying the functor $[-,\mm{BO}_m\langle n\rangle]$ and the Puppe sequence, the null-homotopy of $g\circ \mm{H}|_{\vee_{i=2}^d S^n_i}$ induces the theorem. 
\end{proof}

\section{Proof of Theorem \ref{mainth}}\label{Sepro}


Let $N$ be an $(n-1)$-connected $2n$-Ponincar\'e complex 
$$N\simeq (\vee_{i=1}^d S^n )\cup_\alpha e^{2n}$$
\begin{lemma}\label{factorth}
Let $d\ge 2$ and $m>n\ge 2$. Any map $g:N\to \mm{BO}_m\langle n\rangle$ factors as
$$g:N\stackrel{q}{\to }Q\stackrel{\mm{g}}{\to }\mm{BO}_m\langle n\rangle$$
where $Q= (S^n\bigvee S^n )\bigcup_\beta e^{2n}$ is a Poincar\'e complex.  
\end{lemma}
\begin{proof}
	By Lemma \ref{transform} and the cofibration $\bar N\to N$, we may assume without loss of generality that the restriction $$\bar g=g|_{\bar N}:\bar N=\vee_{i=1}^d S^n_i \to \mm{BO}_m\langle n\rangle$$ satisfies 
	$\bar g(a_i)=0$ for all $i\ge 2$ where $a_i\in H_n(S^n_i)\subset H_n(N)$ denotes the canonical homology generator of the $i$-th $n$-sphere.  Let $\{a_1^\ast,\cdots,a_d^\ast\}\subset H^n(N)$ be the dual basis of $\{a_1,\cdots,a_d\}\subset H_n(N)$ with respect to the Kronecker pairing $\langle a^\ast_i, a_j\rangle=\delta_{ij}$.

Let $[N]\in H_{2n}(N;\mb{Z})$ be a generator.	
	Since $d\ge 2$, we can choose $\gamma\in H^n(N;\mb{Z})$ such that $\gamma\ne a_1^\ast$ and $\langle a_1^\ast\cup \gamma, [N]\rangle=\pm 1$ through the following process: By Poincar\'e duality, there exists $e^\ast$ such that $$\langle a_1^\ast\cup e^\ast, [N]\rangle=\pm 1.$$
Then there are two cases:

Case 1: $e^\ast\ne a_1^\ast$, then let $\gamma=e^\ast$. 

Case 2: $e^\ast= a_1^\ast$, then 
$\langle a_1^\ast\cup a_1^\ast, [N]\rangle=\pm 1$.
Since $d\ge 2$, there exists a generator $\bar e^\ast\in H^n(N;\mb{Z})$ such that $\bar e^\ast\ne a_1^\ast$. If $\langle a_1^\ast\cup \bar e^\ast, [N]\rangle=\pm 1$, then let $\gamma=\bar e^\ast$.  
Otherwise, 
$$\langle a_1^\ast\cup \bar e^\ast, [N]\rangle=\pm k,\quad \text{$k\ne 1$ and $k\ge 0$.}$$
Let $\gamma=(1-k)a_1^\ast+\bar e^\ast$ or $(1+k)a_1^\ast+\bar e^\ast$. Then  
$\langle a_1^\ast\cup \gamma, [N]\rangle=\pm 1$.

By the dual basis, we may assume $\gamma=\Sigma_{i=1}^d k_ia_i^\ast$ for some  $k_i\in \mb{Z}$.
	

Next, we construct a cellular map
	$$\bar q:\vee_{i=1}^d S_i^n\to S_1^n\vee S_2^n$$ 
	The homotopy class of $\bar q$ is determined by its induced homomorphism 
	$$\bar q_\ast:H_n(\vee_{i=1}^d S_i^n;\mb{Z})\to H_n(S_1^n\vee S_2^n;\mb{Z})$$
	Let $b_i$ be the canonical generator of $H_{n}(S_1^n\vee S_2^n;\mb{Z})$.
	We define that $\bar q_\ast(a_1)=b_1+k_1b_2$, $\bar q_\ast(a_i)=k_ib_2$ for $2\le i\le d$. Let $\{b_1^\ast,b_2^\ast\}$ be the dual basis of $\{b_1,b_2\}\subset H_n(S_1^n\vee S_2^n;\mb{Z})$. Then we have
	$$\bar q^\ast(b_2^\ast)=\Sigma_{i=1}^d k_ia_i^\ast=\gamma\quad\text{and} \quad\bar q^\ast(b_1^\ast)=a_1^\ast.$$
	
	We now extend $\bar q$ to the following map 
 $$q:N=(\vee_{i=1}^d S^n_i)\cup_\alpha e^{2n}\to Q=(S_1^n\vee S_2^n)_{\bar q\circ \alpha}e^{2n}$$
Obviously, $q$ induces an isomorphism on $H_{2n}$.	
	 Thus, $Q$ is a Poincar\'e complex. Moreover, The map $q$
 fits into the following commutative diagram of cofibrations:
\[
\xymatrix@C=.9cm{
\vee_{i=2}^d S^n_i\ar[d]^{}\ar[r]^{}&N \ar[r]\ar[d]^{q}&S^n_1\cup e^{2n}\ar[d]\\
S^n_2\ar[r]^{ }&Q\ar[r]^{}& S^n_1\cup e^{2n}
}
\]
The right-hand vertical map is a homology equivalence, and hence a homotopy equivalence, by the Whitehead theorem.
  Apply the functor $[-,\mm{BO}_m\langle n\rangle]$ to the above diagram. By Theorem \ref{basicob} and the right-hand homotopy equivalence, we finish the proof. 
\end{proof}


\begin{proof}[Proof of Item (1) of Theorem \ref{mainth}]
Recall $d\ge 2$ and $m>n\ge 3$.
	By Lemma \ref{liftBOn} and \ref{factorth}, for any rank-$m$ real vector bundle $\xi$ over $N$, there exists a rank-$m$ real vector bundle $\eta$ over $Q$ such that $q^\ast\eta=\xi$. This  induces a bundle morphism
	 between their sphere bundles:
	\[
\xymatrix@C=.9cm{
S^{m-1}\ar[d]^{\mm{id}}\ar[r]^{\mathfrak j}&M\ar[r]\ar[d]^{\mathrm{q}}&N\ar[d]^-{q}\\
S^{m-1}\ar[r]^{\mathfrak i}&E\ar[r]^{\pi}&Q
}
\]

Since $m>n$, the sphere bundle projection $\pi :E\to Q$ has a cross section on the $n$-skeleton of $Q$. In other words, there exists a natural inclusion $i:S^n\vee S^n\to E$ such that the composition $$\pi\circ i:S^n\vee S^n\to S^n\vee S^n\subset Q$$ is the identity.

By the Hilton-Milnor theorem, the canonical inclusion $$j:S^n\vee S^n\to S^n\times S^n$$
has a right homotopy inverse after looping. Precisely, there is a map $$\phi :\Omega(S^n\times S^n)\to \Omega(S^n\vee S^n)$$
such that $\Omega j\circ \phi\simeq \mm{id}$.

Following \cite[Lemma 2.3]{BeTh}, we have that the composite
$$\Omega(S^n\times S^n)\stackrel{\phi}{\to }\Omega(S^n\vee S^n)\stackrel{\Omega i}{\to }\Omega E\stackrel{\Omega \pi}{\to }\Omega Q$$
is a homotopy equivalence. Hence the map $\Omega \pi$ admits a right homotopy inverse $\varphi:\Omega Q\to \Omega E$. Thus the sphere bundle
$$S^{m-1}\stackrel{\mathfrak i}{\to} E\stackrel{\pi}{\to} Q$$
splits after looping to give
$$\Omega E\simeq \Omega S^{m-1}\times \Omega Q.$$
Hence the map $\Omega \mathfrak i$ has a left homotopy inverse
$r:\Omega E\to \Omega S^{m-1}$. Then $r\circ \Omega \mm{q}:\Omega M\to \Omega S^{m-1}$ is a left homotopy inverse of the map $\Omega \mathfrak j$. Therefore, the sphere bundle 
$$S^{m-1}\stackrel{\mathfrak j}{\to} M\to N$$
 splits after looping to give
$\Omega M\simeq \Omega S^{m-1}\times \Omega N$.
\end{proof}

\begin{proof}[Proof of Item (3) of Theorem]
	Since $d=0$, $N\simeq S^{2n}$. The fibration
	$$S^{m-1}\to M\stackrel{\pi }{\to} S^{2n}$$
	induces the following exact sequence
	$$\pi_{2n}(M)\stackrel{\pi_\ast }{\to} \pi_{2n}(S^{2n})\to \pi_{2n-1}(S^{m-1})\to \cdots$$
	When $2n<m$, $\pi_{2n-1}(S^{m-1})=0$. The surjection $\pi_\ast$ implies that $\pi$ has a right homotopy inverse. It follows the sphere
bundle splits after looping to give $\Omega M\simeq \Omega S^{2n}\times \Omega S^{m-1}$. This completes the proof. 
\end{proof}

\begin{Ack}
	The author would like to express sincere gratitude to the referee for many valuable suggestions.
\end{Ack}

\end{document}